\documentclass[leqno,12pt]{amsart}
\usepackage{latexsym}
\usepackage{amssymb,amsmath,amsfonts}
\usepackage{mathtools}
\usepackage[utf8]{inputenc}
\usepackage[usenames,dvipsnames]{color}
\newcommand{ \bl}{\color{blue}}

\newtheorem{theorem}{Theorem}[section]

\newtheorem{lemma}[theorem]{Lemma}

\newtheorem{corollary}[theorem]{Corollary}
\theoremstyle{definition}
\newtheorem{definition}[theorem]{Definition}
\newtheorem{example}{Example}

\newtheorem{remark}{Remark}

\numberwithin{equation}{section}
\usepackage[left=1in,right=1in,top=1in,bottom=1in]{geometry}

\begin{document}
\title[GMEP And A Countable Family of Nonexpansive-Type Maps] 
{ A hybrid scheme for fixed points of a countable family of generalized non-expansive-type maps and generalized mixed equilibrium problems}

\keywords{ {\bl O}ptimization, equilibrium problem, $J_*-$ nonexpansive, fixed points, variational inequality, strong convergence.
**Corresponding author: Markjoe O. Uba (markjoeuba@gmail.com)}

\maketitle

\begin{center}
\vspace{0.3cm}
Peter U. Nwokoro$^{++}$, Maria A. Onyido$^{**}$, Markjoe O. Uba$^{**}$, and Cyril I. Udeani$^{\&\&}$\\

\vspace{0.3cm}
$^{**}$Department of Mathematical Science, \\
Northern Illinois University,\\
DeKalb, IL 60115, USA.\\

\vspace{0.3cm}
$^{\&\&}$Department of Computer Science, \\
University of Nevada, Las Vegas\\
Las Vegas, NV 89154, USA.\\

\vspace{0.3cm}
$^{++}$Department of Mathematics, \\
University of Nigeria,\\
Nsukka, Nigeria.\\

\vspace{0.4cm} 
\end{center}

\noindent {\bf Abstract.} 
Let $\Omega$ be a nonempty closed and convex subset of a uniformly smooth and uniformly convex real Banach space $\mathcal{X}$ 
with dual space $\mathcal{X}^*$. This article presents a hybrid algorithm for finding a common element of the set of
solutions to a generalized mixed equilibrium problem and the set of common fixed points of a family of a general class of nonlinear nonexpansive maps. The results obtained were employed to study optimization problem. Our results and its applications complement, generalize, and extend several results in literature.

\section{introduction}
\noindent Let $\mathcal{X}$ be a uniformly convex and uniformly smooth real Banach space with dual space $\mathcal{X}^*$. Let $\Omega$ be a nonempty closed and convex subset of $\mathcal{X}$ such that $J\Omega$ is closed and convex, where $J : \mathcal{X} \rightarrow \mathcal{X}^*$ is the normalized duality map on $\mathcal{X}$. Let $\langle \cdot, \cdot \rangle$ denote the duality pairing between $\mathcal{X}$ and $\mathcal{X}^*$. Let $f : J\Omega\times J\Omega \rightarrow \mathbb{R}$ be a bifunction, $\varphi : J\Omega \rightarrow \mathbb{R}$ be a real-valued function, and $A : \Omega \rightarrow \mathcal{X}^*$ be
a nonlinear mapping. The {\it generalized
mixed equilibrium problem} is to find an element $v \in \Omega$ such that
\begin{equation}
    f(Jv,Ju) + \varphi(Ju) - \varphi(Jv) + \langle Av,u-v \rangle \ge 0 ~\forall u \in \Omega.
\end{equation}
In this paper, we consider equilibrium problem such that $f$ is given by 
\begin{equation}
    f(Jx,Jy) = \sum_{i=1}^k f_i(Jx,Jy), ~\forall x,y \in \Omega,
\end{equation}
where $f_i : J\Omega\times J\Omega \rightarrow \mathbb{R}$ {\bl is a bifunction} for each $i\in \{1,2,3,\cdots,k\}, \ k\geq1$. Consequently, we study the following well known generalized
mixed equilibrium problem: find $v \in \Omega$ such that 
\begin{equation}\label{gmep}
    \sum_{i=1}^k f_i(Jv,Ju) + \varphi(Ju) - \varphi(Jv) + \langle Av,u-v \rangle \ge 0 ~\forall u \in \Omega.
\end{equation}
The set of solution of \eqref{gmep} is given by
\begin{equation*}\label{gmepsol}
   GMEP(f,A,\varphi) = \Bigg\{v\in \Omega : \sum_{i=1}^k f_i(Jv,Ju) + \varphi(Ju) - \varphi(Jv) + \langle Av,u-v \rangle \ge 0 ~\forall u \in \Omega\Bigg\}.
\end{equation*}
\noindent If $A=0$, \eqref{gmep} reduces to finding $v \in \Omega$ such that 
\begin{equation}\label{mep}
    \sum_{i=1}^k f_i(Jv,Ju) + \varphi(Ju) - \varphi(Jv) \ge 0 ~\forall u \in \Omega,
\end{equation}
which is called the mixed equilibrium problem and MEP denotes the set of solutions to \eqref{mep}. The class of generalized mixed equilibrium problems is famous, well studied, and contains, as special cases, numerous important classes of nonlinear problems such
as equilibrium problems, optimization problems, variational inequality problems (see e.g., \cite{bcsk, dooai, gsao, lwz} and the references contained
in them).\\

\noindent For an arbitrary real normed space $\mathcal{X}$ with dual space $\mathcal{X}^*,$ an operator $A:\mathcal{X}\rightrightarrows \mathcal{X}^*$
is called {\it monotone} if
\begin{equation}
 \langle \xi-\tau,x-y \rangle \geq0\,\, \forall \,\, \xi \in Ax, \tau \in Ay.
\end{equation}

\noindent It is well known that these operators appear in a wide variety of contexts since they can be found in many functional equations. Many of them also are known to appear in calculus of variations as subdifferential of convex functions. Additionally, monotone operators are well known to have a strong connection with optimization. There are several ongoing research efforts to develop and employ fixed point techniques to approximate solution of
the equation $Au = 0$ when $A$ is monotone. One of such efforts lead to the introduction and study of a new notion of fixed points for maps from $\mathcal{X}$ to $\mathcal{X}^*$ called $J-$fixed points (see e.g., \cite{buoou, ceez, cidu, cuuoi} and the references contained in them).\\

\noindent It is our purpose in this paper to introduce and study a new hybrid algorithm and prove a strong convergence theorems for obtaining a common element in the solutions of a generalized mixed
equilibrium problem and common fixed points for a countable family of generalized $J_*-$nonexpansive maps (as well as generalized $J-$nonexpansive maps) in a uniformly smooth and uniformly convex real Banach space. The results introduced are applicable in classical Banach spaces such as $L_{p},~l_{p}, or~ W^{m}_{p}(\Omega)$, $p \in (1,\infty)$, where $W^{m}_{p}(\Omega)$ denotes the usual Sobolev space. Additionally, the Hilbert space case of the results obtained complement, extend and improve
several results in literature.

\bigskip
\section{ Preliminaries}
\noindent In this section, we present definitions and lemmas which we shall use in proving our main result. 
\begin{definition}{\bf (Uniformly Smooth)}
A normed linear space $\mathcal{X}$  of dimension $\geq 2$ is called {\it uniformly smooth} if the modulus of smoothness $\rho_{\mathcal{X}}:[0,\infty)\rightarrow [0,\infty)$, defined by
\[ \rho_{\mathcal{X}(\tau )}:= \sup\left\{\frac{\| x+y\| +\| x-y\|}{2}-1: \| x\| =1, \| y\| =\tau,~~\tau >0\right\}  \]
tends to $0$ as $\tau \to 0.$  
\end{definition}
\begin{definition}{\bf (Uniformly Convex)}
 A Banach space $\mathcal{X}$  is \textit{uniformly convex} if the \textit{modulus of convexity} of $\mathcal{X}, \; \delta_{\mathcal{X}}:(0,2]\rightarrow [0,1]$ defined by
 $$\delta_{\mathcal{X}(\epsilon)}:=\inf\Big\{1-\Big\|\frac{x+y}{2}\Big\|:\|x\|=\|y\|=1;\, \epsilon=\|x-y\|\Big\}$$
 is positive for every $\epsilon\in (0,2].$
\end{definition}
\begin{definition}\label{ndm}{\bf (Normalized duality map)}
 The map $J:\mathcal{X}\rightarrow 2^{\mathcal{X}^*}$ defined by
$$Jx:=\big \{x^*\in \mathcal{X}^*:\big <x,x^*\big >=\|x\|.\|x^*\|,~\|x\|=\|x^*\|\big \}$$ is called the {\it normalized duality map} on $\mathcal{X}$.
\end{definition}
 
\noindent It is well known that if $\mathcal{X}$ is smooth, strictly convex and reflexive then $J^{-1}$ exists 
(see e.g., \cite{tak}); 
$J^{-1}:\mathcal{X}^{*}\rightarrow \mathcal{X}$ is the normalized duality mapping on $\mathcal{X}^{*}$,
and $J^{-1}=J_{*}, ~JJ_{*}=I_{\mathcal{X}^{*}}$ and $J_{*}J =I_{\mathcal{X}}$, where $I_{\mathcal{X}}$ and $I_{\mathcal{X}^{*}}$ are the identity maps on
 $\mathcal{X}$ and $\mathcal{X}^{*}$, respectively.  A well known property of $J$ is  (see e.g., \cite{Io, tak}):
\begin{itemize}
\item If $\mathcal{X}$ is uniformly smooth, then $J$ is uniformly continuous on bounded subsets of $\mathcal{X}$.
\end{itemize}

\begin{definition}{\bf (Lyapunov Functional)}
Let $\mathcal{X}$ be a smooth real Banach space with dual $\mathcal{X}^*$. The {\it Lyapounov functional} $\phi:\mathcal{X}\times \mathcal{X}\to\mathbb{R}$, is defined by
\begin{eqnarray}\label{Lya}
 \phi(x,y)=\|x\|^2-2\langle x,Jy\rangle+\|y\|^2,~~\text{for}~x,y\in \mathcal{X},
\end{eqnarray}
where $J$ is the normalized duality map.
\end{definition}

\noindent The {\it Lyapounov functional} was introduced and studied in \cite{b1, AG, b2, 40}. If $\mathcal{X}=H$, a real Hilbert space, then
equation (\ref{Lya}) reduces to $\phi(x,y)=\|x-y\|^2$ for $x,y\in H.$ It is obvious from the definition of the function $\phi$ that
\begin{eqnarray}\label{fi}
 (\|x\|-\|y\|)^2\leq \phi(x,y)\leq(\|x\|+\|y\|)^2~~\text{for}~x,y\in \mathcal{X}.
\end{eqnarray}

\begin{definition}\label{gen nonexp}{\bf (Generalized nonexpansive)}
Let $\Omega$ be a nonempty closed and convex subset of a real Banach space $\mathcal{X}$ and $T$ be a map from $\Omega$ to $\mathcal{X}$. The map $T$ is called 
{\it generalized nonexpansive} if $F(T):=\{x\in \Omega: Tx=x\}\neq \emptyset$ and $\phi(Tx,p)\leq \phi(x,p)$ for all $x\in \Omega, p\in F(T)$. 
\end{definition}
\begin{definition}{\bf (Retraction)}
A map $R$ from $\mathcal{X}$ onto $\Omega$ is said to be a retraction if $R^{2}=R$. The map $R$ is said to be {\it sunny}  if $R(Rx+t(x-Rx))=Rx$ 
for all $x\in \mathcal{X}$ and {\bl $t\geq 0$}. 
\end{definition}

\noindent A nonempty closed subset $\Omega$ of a smooth Banach space $\mathcal{X}$ is said to be a {\it sunny generalized nonexpansive retract} of $\mathcal{X}$ if there exists 
a sunny generalized nonexpansive retraction $R$ from $\mathcal{X}$ onto $\Omega$.\\

\noindent {\bf NST-condition.}
Let $\Omega$ be a closed subset of a Banach space $\mathcal{X}$. Let $\{T_{n}\}$ and $\Gamma$ be two families of generalized nonexpansive maps of $\Omega$ into 
$\mathcal{X}$ such that $\cap_{n=1}^{\infty}F(T_{n})=F(\Gamma)\neq \emptyset,$ where $F(T_{n})$ is the set of fixed points of $\{T_{n}\}$ and $F(\Gamma)$
is the set of common fixed points of $\Gamma$. 

\begin{definition}
The sequence  $\{T_{n}\}$ satisfies the NST-condition (see e.g., \cite{nakajostak}) with $\Gamma$ if for each bounded sequence $\{x_{n}\}\subset \Omega$,
$$\lim_{n\rightarrow \infty}||x_{n}-T_{n}x_{n}||=0 \Rightarrow  \lim_{n\rightarrow \infty}||x_{n}-Tx_{n}||=0, ~for ~all~T\in \Gamma.$$
\end{definition}

\begin{remark}\label{rmk1}
  If $\Gamma =\{T\}$ a singleton, $\{T_{n}\}$ satisfies the NST-condition with $\{T\}$. If $T_{n}=T$ for all $n\geq 1$, 
then, $\{T_{n}\}$ satisfies the NST-condition with $\{T\}$.
\end{remark}

\noindent Let $\Omega$ be a nonempty closed and convex  subset of a uniformly smooth and uniformly convex real Banach space $\mathcal{X}$ with dual space $\mathcal{X}^*$.
Let $J$ be the normalized duality map on $\mathcal{X}$ and $J_{*}$ be the normalized duality map on $\mathcal{X}^*$. Observe that under this setting, $J^{-1}$ exists 
and $J^{-1}=J_{*}$. With these notations, we have the following definitions.

\begin{definition} {\bf (Closed map)}\cite{uoo}
 A map $T:\Omega\rightarrow \mathcal{X}^*$ is called {\it $J_{*}-$closed} if $(J_{*}oT) : \Omega\rightarrow \mathcal{X}$ is a closed map, i.e., if $\{x_{n}\}$ is a sequence in $\Omega$ 
such that $x_{n}\rightarrow x$ and  $(J_{*}oT)x_{n}\rightarrow y$, then $(J_{*}oT)x =y$. 
\end{definition}

\begin{definition}{\bf ($J-$fixed Point)}\cite{cidu}
A point $x^*\in \Omega$ is called a {\it $J-$fixed point of $T$} if $Tx^*=Jx^*$. The set of $J-$fixed points of $T$ will be denoted by $F_{J}(T)$.
\end{definition}

\begin{definition}{\bf (Generalized $J_{*}-$nonexpansive)}\cite{uoo}
 A map $T:\Omega\rightarrow \mathcal{X}^*$ will be called {\it generalized $J_{*}-$nonexpansive} if $F_{J}(T)\neq \emptyset$, and 
$\phi (p, (J_{*}oT)x) \leq \phi(p,x)$ for all $x\in \Omega$ and for all $p\in F_{J}(T)$.
\end{definition}

\begin{remark}\label{rmk2}
Exampes of generalized $J_{*}-$nonexpansive maps in Hilbert and more general Banach spaces were given in \cite{ceez} and \cite{uoo}.
\end{remark}

\noindent Let $\Omega$ be a nonempty closed subset of a smooth, strictly convex and reflexive Banach space $\mathcal{X}$ such that $J\Omega$ is closed and convex. For solving equilibrium problem, let us assume that a bifunction $f:J\Omega\times J\Omega\rightarrow \mathbb{R}$ satisfies the following conditions:
\begin{itemize}
 \item[(A1)]  $f(x^*,x^*)=0$ for all $x^*\in J\Omega$;
 \item[(A2)]  $f$ is monotone, i.e. $f(x^*,y^*)+f(y^*,x^*)\leq0$ for all $x^*,y^*\in J\Omega$;
 \item[(A3)]  for all $x^*,y^*,z^*\in J\Omega$, $\limsup_{t\downarrow0} f(tz^*+(1-t)x^*,y^*)\leq f(x^*,y^*)$;
 \item[(A4)]  for all $x^*\in J\Omega$, $f(x^*,\cdot)$ is convex and lower semicontinuous. 
\end{itemize}

\noindent With the above definitions, we now provide the lemmas we shall use.

\begin{lemma}\cite{Zh}\label{zhang}
Let $\mathcal{X}$ be a uniformly convex Banach space, $r > 0$ be a positive number, and $B_r(0)$ be a closed ball of $\mathcal{X}$. For any given points $\{ x_1, x_2, \cdots , x_N \} \subset B_r(0)$ and any given positive numbers $\{ \lambda_1, \lambda_2, \cdots , \lambda_N \}$ with $\sum_{n = 1} ^{N} \lambda_n = 1,$ there exists a continuous strictly increasing and convex function $g : [0, 2r) \to [0, \infty)$ with $g(0) = 0$ such that, for any $i,j \in \{ 1,2, \cdots N \}, \; i < j,$
\begin{equation}
    \| \sum_{n = 1} ^{N} \lambda_n x_n \|^2 \le \sum_{n = 1} ^{N} \lambda_n\|x_n\|^2 - \lambda_i \lambda_j g(\|x_i - x_j\|).
\end{equation}
\end{lemma}

\begin{lemma}\cite{b2}\label{man}
 Let $\mathcal{X}$ be a real smooth and uniformly convex Banach space, and let $\{x_n\}$ and $\{y_n\}$ be two sequences of $\mathcal{X}$. If either $\{x_n\}$ or $\{y_n\}$
 is bounded and $\phi(x_n,y_n)\to0$ as $n\to \infty$, then $\|x_n-y_n\|\to0$ as $n\to\infty$.\label{bd}
\end{lemma}

\begin{lemma}\cite{b1}\label{lem1}
Let $\Omega$ be a nonempty closed and convex subset of a smooth, strictly convex and reflexive Banach space $\mathcal{X}$. Then, the following are equivalent.\\
$(i)$ $\Omega$ is a sunny generalized nonexpansive retract of $\mathcal{X}$,\\
$(ii)$ $\Omega$ is a generalized nonexpansive retract of $\mathcal{X}$,\\
$(iii)$ $J\Omega$ is closed and convex.
\end{lemma}

\begin{lemma}\cite{b1}\label{lem2}
Let $\Omega$ be a nonempty closed and convex subset of a smooth and  strictly convex Banach space $\mathcal{X}$ such that there exists a sunny generalized 
nonexpansive retraction $R$ from $\mathcal{X}$ onto $\Omega$.Then, the following hold.\\
$(i)$ $z=Rx$ iff $\langle x-z, Jy-Jz\rangle \leq 0$ for all $y\in \Omega$,\\
$(ii)$ $\phi (x,Rx) + \phi (Rx, z) \leq \phi (x,z)$ {\bl for all $z\in \Omega$}.  
\end{lemma}

\begin{lemma}\cite{itaka1}\label{ww1}
 Let $\Omega$ be a nonempty closed sunny generalized nonexpansive retract of a smooth and strictly convex Banach space $\mathcal{X}$. Then the sunny generalized nonexpansive
 retraction from $\mathcal{X}$ to $\Omega$ is uniquely determined.
\end{lemma}

\begin{lemma}\cite{Blum}\label{ww3}
 Let $\Omega$ be a nonempty closed subset of a smooth, strictly convex and reflexive Banach space $\mathcal{X}$ such that $J\Omega$ is closed and convex, let $f$ be a bifunction
 from $J\Omega\times J\Omega$ to $\mathbb{R}$ satisfying $(A1)-(A4)$. For $r>0$ and let $x\in E$. Then there exists $z\in \Omega$ such that
 $f(Jz,Jy)+\frac{1}{r}\langle z-x, Jy-Jz\rangle\geq0,~~\forall~~y\in \Omega.$
\end{lemma}

\begin{lemma}\cite{ttaka}\label{ww4}
 Let $\Omega$ be a nonempty closed subset of a smooth, strictly convex and reflexive Banach space $\mathcal{X}$ such that $J\Omega$ is closed and convex, let $f$ be a bifunction
 from $J\Omega\times J\Omega$ to $\mathbb{R}$ satisfying $(A1)-(A4)$. For $r>0$ and let $x\in \mathcal{X}$, define a mapping $T_r(x):\mathcal{X}\rightarrow \Omega$ as follows:
 $$T_r(x)={\bl \Bigg\{}z\in \Omega:\sum_{i=1}^k f_i(Jz,Jy)+\langle y-z, Az\rangle+\varphi(y)-\varphi(z)+\frac{1}{r}\langle y - z, Jz-Jx\rangle\geq0,~~\forall~~y\in \Omega{\bl \Bigg\}}.$$ Then the following hold:
 \begin{itemize}
 \item[(i)]  $T_r$ is single valued;
 \item[(ii)]  for all $x,y\in \mathcal{X}$, $\langle T_rx-T_ry, JT_rx-JT_ry\rangle\leq\langle x-y, JT_rx-JT_ry\rangle$;
 \item[(iii)] $F(T_r)=GMEP(f,A,\varphi)$;
 \item[(iv)] $\phi (p,T_r(x)) + \phi (T_r(x), x) \leq \phi (p,x)$ for all $p\in F(T_r)$, $x\in \mathcal{X}$.
 \item[(v)]  $GMEP(f,A,\varphi)$ is closed and $JGMEP(f,A,\varphi)$ is closed and convex. 
\end{itemize}
\end{lemma}

\begin{lemma}\cite{uoo}\label{lemma1}
Let $\mathcal{X}$ be a uniformly convex and uniformly smooth  real Banach space with dual space $\mathcal{X}^*$ and let $\Omega$ be a closed subset of $\mathcal{X}$ such that $J\Omega$ is closed and convex. Let $T$ be a
generalized $J_{*}-$nonexpansive map from $\Omega$ to $\mathcal{X}^*$ such that $F_{J}(T) \neq \emptyset$, then $F_{J}(T)$ and $JF_{J}(T)$ are closed.
\end{lemma}

\begin{lemma}\cite{uoo} \label{lemma2}
Let $\mathcal{X}$ be a uniformly smooth and uniformly convex real Banach space and let $\Omega$ be a closed subset of $\mathcal{X}$ such that $J\Omega$ is closed and convex. Let $T$ be a
generalized $J_{*}-$nonexpansive map from $\Omega$ to $\mathcal{X}^*$ such that $F_{J}(T) \neq \emptyset$. If $JF_{J}(T)$ is convex, then $F_{J}(T)$ is a sunny generalized nonexpansive retract of $\mathcal{X}$.
\end{lemma}
\begin{example}
Let $\mathcal{X}$ be a uniformly smooth and uniformly convex real Banach space  with dual space $\mathcal{X}^*$ and let $\Omega$ be a nonempty closed subset of $\mathcal{X}$. Let $T:\Omega\rightarrow \mathcal{X}^*$, be a generalized $J_{*}-$nonexpansive maps such that $F_{J}(T)\neq \emptyset$. Let $\alpha_n  \subset (0,1)$ such that $1 - \alpha_n \ge \frac{1}{2}$. For all $n \in \mathbb{N}$, define $T_{n}:\Omega\rightarrow \mathcal{X}^*$ by 
\begin{equation}
    T_{n}u = \alpha_nJu + (1 - \alpha_n)Tu, ~\forall~ u\in \Omega.
\end{equation}
Then $\{T_n\}$ is a countable family of generalized $J_{*}-$nonexpansive maps satisfying $NST-$condition with $T$.
\end{example}
\begin{proof}
Clearly $F_{J}(T_{n})=F_{J}(T)$ $\forall~n \in \mathbb{N}$. Hence, $\cap_{n=1}^{\infty}F_{J}(T_{n})=F_{J}(T)$. For $u \in \Omega$, $v \in F_{J}(T_{n})$, 
\begin{eqnarray*}
    \phi(v, J_{*}oT_{n}u)& = &\phi(v, J_{*}(\alpha_nJu + (1 - \alpha_n)Tu))   \nonumber\\
                        & = & ||v||^2 -2\langle v, (\alpha_nJu + (1 - \alpha_n)J(J_{*}oT)u)\rangle + ||\alpha_nJu + (1 - \alpha_n)J(J_{*}oT)u||^2 \nonumber\\
                        & = & ||v||^2 -2\alpha_n\langle v, Ju\rangle -2(1 - \alpha_n)\langle v, J(J_{*}oT)u)\rangle + ||\alpha_nJu + (1 - \alpha_n)J(J_{*}oT)u||^2 \nonumber\\
                        & \le & ||v||^2 -2\alpha_n\langle v, Ju\rangle -2(1 - \alpha_n)\langle v, J(J_{*}oT)u)\rangle + \alpha_n||u||^2 + (1 - \alpha_n)||J_{*}oTu||^2 \nonumber\\
                        & = & \alpha_n||v||^2 -2\alpha_n\langle v, Ju\rangle + \alpha_n||u||^2 + (1 - \alpha_n)||v||^2 \nonumber\\
                        &&-2(1 - \alpha_n)\langle v, J(J_{*}oT)u)\rangle + (1 - \alpha_n)||J_{*}oTu||^2 \nonumber\\
                         & = & \alpha_n\phi(v, u) + (1 - \alpha_n)\phi(v, J_{*}oTu) \nonumber\\
                        & \le & \phi(v, u).
\end{eqnarray*}
Hence, $\{T_n\}$ is a countable family of generalized $J_{*}-$nonexpansive maps.\\

\noindent Let $\{u_n\}$ be a bounded sequence in $\Omega$ such that $\lim ||Ju_n - T_nu_n|| = 0$. This implies that $\{J_{*}oTu_n\}$ is bounded. From the definition of $T_n$, we obtain the following inequality
\begin{equation}
   ||Ju_n - Tu_n|| = \frac{1}{(1 - \alpha_n)} ||Ju_n - T_nu_n|| \le 2||Ju_n - T_nu_n||.
\end{equation}
This shows that $\lim ||Ju_n - Tu_n|| = 0$.
\end{proof}
\section{Main Results}
\noindent We now prove the following theorem.
\begin{theorem}\label{main}
Let $\mathcal{X}$ be a uniformly smooth and uniformly convex real Banach space  with dual space $\mathcal{X}^*$ and let $\Omega$ be a nonempty closed and convex subset of $\mathcal{X}$ 
such that $J\Omega$ is closed and convex. Let $\varphi:J\Omega \rightarrow \mathbb{R}$ be a lower semi-continuous and convex function. For each $i \in \{1, 2, 3, ..., N\}$, {\bl let $f_{i}$ be a bifunction} from $J\Omega\times J\Omega$ to $\mathbb{R}$ satisfying $(A1)-(A4)$, 
$T^i_{n}:\Omega\rightarrow \mathcal{X}^*, n=1, 2, 3, ...$ be an infinite family of generalized $J_{*}-$nonexpansive maps and $\Gamma$ be a family of 
closed and generalized $J_{*}-$nonexpansive maps from $\Omega$ to $\mathcal{X}^*$ such that $\cap_{n=1}^{\infty}F_{J}(T^i_{n})=F_{J}(\Gamma) \neq \emptyset$
and $B := F_{J}(\Gamma)\cap GMEP(f,A, \varphi) \neq \emptyset.$ Assume that $JF_{J}(\Gamma)$ is convex and 
$\{T^i_{n}\}$ satisfies the NST-condition with $\Gamma$.
Let $\{x_{n}\}$ be generated by:  

\begin{equation}\label{algcoro1}
\begin{cases}  & x_{1} = x\in \Omega; \Omega_{1}=\Omega, \cr
                     & y_n =J^{-1}(\alpha^0_{n}Jx_{n} + \sum_{i = 1} ^{N} \alpha^i_{n}T^i_{n}x_{n}),\cr
                    & u_{n}\in \Omega, ~~such~~that~~ \sum_{i = 1} ^{N} f_i(Ju_n,Jy)+\varphi(Jy) - \varphi (Ju_n) 
                    \cr & + \langle y - u_n, A u_n \rangle + \frac{1}{r_n}\langle u_n-y_{n}, Jy-Ju_n\rangle\geq0,~~\forall~~y\in \Omega,\cr
                    & \Omega_{n+1} =\{z\in \Omega_{n} : \phi(z, u_{n}) \leq \phi(z, x_{n})\},\cr
                     & x_{n+1} = R_{\Omega_{n+1}}x,
\end{cases} 
\end{equation}

for all $n\in \mathbb{N}, \; \{\alpha^i_{n}\}\in [0,1]$ such that $\sum_{i = 0} ^{N} \alpha^i_{n} = 1$,  $\{r_n\}\subset [a,\infty)$
 for some $a>0$. Then, $\{x_{n}\}$ converges strongly 
to $R_B x$, where $R_B$ is the sunny generalized $J_*-$nonexpansive retraction of $\mathcal{X}$ onto $B$.   
\end{theorem}

\noindent {\bf Sketch of proof:}
We show that 
\begin{itemize}
    \item [\bf I] $\{x_{n}\}$ is well defined;
    \item [\bf II] $F_{J}(\Gamma)\cap GMEP(f,A, \varphi) \subset \Omega_{n} \text{ for all} ~ n \ge 1$;
    \item [\bf III] {\bl $R_B x$ exists} as a point in $\Omega_n\; \text{for all} ~ n \ge 1$;
    \item [\bf IV] $x_n \to x^*$ for some $x^* \in \Omega$;
    \item [\bf V] $x^*\in F_{J}(\Gamma)\cap GMEP(f,A, \varphi)$;
    \item [\bf VI] $x^* = R_{B}x$.
\end{itemize}

\begin{proof}
 The proof is given in $6$ steps.
 
\noindent {\bf (I)} It is easy to see that $J\Omega_n$ is closed and convex for each $n\geq1$. Therefore, from 
Lemma \ref{lem1}, we have that $\Omega_n$ is a sunny generalized $J_*-$nonexpansive retract of $\mathcal{X}$ for each $n\geq1$. Hence, $\{x_{n}\}$ is well defined.\\

\noindent {\bf (II)} 
Clearly, $B \subset \Omega_1$. Suppose  $B \subset \Omega_{n}$ for some $n\in \mathbb{N}$. Let $u\in B$, and $ u_n=T_{r_n}y_n$ for all $n\in \mathbb{N}$. Using the fact that {\bl$\{T^i_{n}\}$}
is an infinite family of generalized $J_*-$nonexpansive maps, the definition of $y_n$, Lemmas \ref{ww4}, and  \ref{zhang}, we compute as follows:
\begin{eqnarray}\label{222}
 \phi(u, u_{n}) & = & \phi(u,T_{r_n}y_n)\leq\phi(u,y_n)  = \phi{\bl\Big(}u, J^{-1}{\bl \Big(}\alpha^0_{n}Jx_{n} + \sum_{i = 1} ^{N}                         \alpha^i_{n}T^i_{n}x_{n}{\bl\Big)\Big)}\nonumber\\
                & = & ||u||^2 -2{\bl\Big\langle} u, \alpha^0_{n}Jx_{n} + \sum_{i = 1} ^{N}                         \alpha^i_{n}T^i_{n}x_{n}{\bl\Big\rangle} + {\bl \Big|\Big|}\alpha^0_{n}Jx_{n} + \sum_{i = 1} ^{N}                         \alpha^i_{n}T^i_{n}x_{n}{\bl\Big|\Big|}^2 \nonumber\\
                & \leq & ||u||^2 -2\alpha^0_{n}\langle u, Jx_{n}\rangle + \alpha^0_{n}||x_{n}||^2 -2{\bl\Big\langle} u, \sum_{i = 1} ^{N} \alpha^i_{n}T^i_{n}x_{n}{\bl\Big\rangle}  + \sum_{i = 1} ^{N} \alpha^i_{n}||T^i_{n}x_{n}||^2 \nonumber\\
                && -  \alpha^0_{n}\alpha^i_{n}g(||Jx_{n}- T^i_{n}x_{n}||)\nonumber\\
                & = & \alpha^0_{n}(||u||^2 -2\langle u, Jx_{n}\rangle + ||x_{n}||^2) + \sum_{i = 1} ^{N}\alpha^i_{n}(||u||^2 -2\langle u, J(J_{*}oT^i_{n})x_{n}\rangle  + ||T^i_{n}x_{n}||^2) \nonumber\\ 
                && -  \alpha^0_{n}\alpha^i_{n}g(||Jx_{n}- T^i_{n}x_{n}||)\nonumber\\
                & = & \alpha^0_{n}\phi(u, x_{n}) + \sum_{i = 1} ^{N}\alpha^i_{n}\phi(u, J_{*}oT^i_{n}x_{n}) -  \alpha^0_{n}\alpha^i_{n}g(||Jx_{n}- T^i_{n}x_{n}||)\nonumber\\
                & \leq & \alpha^0_{n}\phi(u, x_{n}) + \sum_{i = 1} ^{N}\alpha^i_{n}\phi(u, x_{n}) -  \alpha^0_{n}\alpha^i_{n}g(||Jx_{n}- T^i_{n}x_{n}||).\nonumber
\end{eqnarray}
Hence, we obtain the following inequality for each $i \in \{1,2,3,\cdots, N\}$
\begin{equation}\label{key}
\phi(u, u_{n}) \leq \phi(u, x_{n}) -  \alpha^0_{n}\alpha^i_{n}g(||Jx_{n}- T^i_{n}x_{n}||).
\end{equation}
Additionally, we have that 
\begin{equation}\label{key2}
\phi(u, y_{n}) \leq \phi(u, x_{n}) ~\forall~ n\in \mathbb{N}.
\end{equation}
From inequality (\ref{key}), we conclude that $u\in \Omega_{n+1}$. Thus, $B \subset{\Omega_n} $ for all $n \ge 1$.\\

\noindent{\bf (III)}
From Lemma \ref{ww4} $JGMEP(f,A, \varphi)$ is closed and convex. Also, using our assumption and {\bl Lemma} \ref{lemma1}, we have that  {\bl $J(F_{J}(\Gamma))$} is closed and convex.  
Since $\mathcal{X}$ is uniformly convex, $J$ is one-to-one. Thus, we have that,

\overfullrule = 0mm\hbox to 10pt {
$J\Big(F_{J}(\Gamma)\cap GMEP(f,A, \varphi)\Big)= JF_{J}(\Gamma)\cap JGMEP(f,A, \varphi)$} 

\noindent and so $J(B)$ is closed and convex. Hence, using Lemma \ref{lem1}, we obtain that $B$ is a sunny generalized $J_*-$nonexpansive retract of $\mathcal{X}$. Thus, using {\bl Lemma} \ref{ww1}, we have that $R_Bx$ exists as a point in $\Omega_n$ for all $n \ge 1$. \\

\noindent{\bf (IV)} Using the fact that $x_n=R_{\Omega_n}x$ and {\bl Lemma} \ref{lem2}$(ii)$, we obtain
$$\phi(x,x_{n})=\phi(x,R_{\Omega_{n}}x)\leq \phi(x,u),$$
for all $u\in F_{J}(\Gamma)\cap GMEP(f,A, \varphi)\subset \Omega_{n}.$
This implies that $\{\phi(x,x_{n})\}$ is bounded. Hence, from equation \eqref{fi}, $\{x_{n}\}$ is bounded. Also, since  
 $x_{n+1} = R_{\Omega_{n+1}}x \in \Omega_{n+1} \subset \Omega_n$, and $x_n=R_{\Omega_n}x \in \Omega_n$, applying Lemma \ref{lem2}$(ii)$ gives

$$\phi(x,x_{n})\leq \phi(x,x_{n+1}) ~\forall~ n\in \mathbb{N}.$$ So, $\lim _{n\rightarrow \infty}\phi(x, x_{n})$ exists.
Again, using Lemma \ref{lem2}$(ii)$ and $x_{n}=R_{\Omega_{n}}x$, we obtain that for all $m,n\in \mathbb{N}$ with $m>n$, 
\begin{eqnarray}
\phi(x_{n},x_{m})& =&\phi(R_{\Omega_{n}}x,x_{m})\leq  \phi(x,x_{m}) -\phi(x,R_{\Omega_{n}}x) \nonumber\\
&= & \phi(x,x_{m}) -\phi(x,x_{n}) \rightarrow 0~as~ n\rightarrow \infty.
\end{eqnarray} From Lemma \ref{man}, we conclude that $||x_{n}-x_{m}||\rightarrow 0, ~as~ m, ~ n\rightarrow \infty.$ Hence,
$\{x_{n}\}$ is a Cauchy sequence in $\Omega$, and so, there exists $x^*\in \Omega$ such that $x_{n}\rightarrow x^*$.\\

\noindent{\bf (V)} From the definitions of $\Omega_{n+1}$ 
and $x_{n+1}$, we obtain that $\phi(x_{n+1}, u_{n})\leq \phi(x_{n+1},x_{n})\rightarrow 0$ as $n\rightarrow \infty.$ Hence, by Lemma \ref{man} , we have that
\begin{equation}\label{limun}
\underset{n \to \infty}{\text{\bl lim}}||x_{n}-u_{n}||= 0.    
\end{equation}

 \noindent Since $x_n\rightarrow x^*~as~n\rightarrow \infty$, equation (\ref{limun}) implies that $u_n\rightarrow x^*~as~n\rightarrow \infty$. Observe that since $J$ is uniformly continuous on bounded subsets of $\mathcal{X}$, it follows from (\ref{limun}) that 
 \begin{equation}\label{limjun}
\lim_{n\rightarrow \infty}||Ju_{n}-Jx_{n}||=0.     
\end{equation}

\noindent From inequality (\ref{key}) and the fact that $g$ is nonnegative, we obtain
$$0 \leq \alpha^0_{n}\alpha^i_{n}g(||Jx_{n}- T^i_{n}x_{n}||) \leq \phi(u, x_{n}) - \phi(u, u_{n}) \leq 2||u||.||Jx_{n}-Ju_{n}|| + ||x_{n}-u_{n}||M,$$

\noindent for some $M > 0.$ Let $ \liminf \alpha^0_{n}\alpha^i_{n} = a$. Since $a > 0$, there exists $n_0\in\mathbb{N}$:
$$0<\frac{a}{2}<\alpha^0_{n}\alpha^i_{n}~\textrm{for all}~ n\ge n_0.$$ Thus, $$0 \leq \frac{a}{2}g(||Jx_{n}- T^i_{n}x_{n}||) \leq 2||u||.||Jx_{n}-Ju_{n}|| + ||x_{n}-u_{n}||M ~\textrm{for all}~ n\ge n_0.$$
\noindent Thus, from (\ref{limun}), (\ref{limjun}), and properties of $g$, we obtain that $\lim_{n\rightarrow \infty}||Jx_{n}-T^i_{n}x_{n}||\ = 0$. Since $\{T^i_{n}\}_{n=1}^{\infty}$ satisfies the NST condition with $\Gamma$, we have that
\begin{equation}
 \lim _{n\rightarrow \infty}||Jx_{n}-T^ix_{n}||=0 ~\forall~ T^i\in \Gamma. 
\end{equation}
  
\noindent Now, since we have established that $x_{n}\rightarrow {\bl x^*\in \Omega}$. Assume that  $(J_{*}oT^i)x_{n}\rightarrow y^*$. Since $T^i$ is closed, we have $y^*=(J_{*}oT^i)x^*$.
Furthermore, by the uniform continuity of $J$ on bounded subsets of $\mathcal{X}$, we have:
$Jx_{n} \rightarrow Jx^*$ and $J(J_{*}oT^i)x_{n}\rightarrow Jy^*$ as $n\rightarrow \infty.$ Hence, we have
$$ \lim_{n\rightarrow \infty}||Jx_{n}-J(J_{*}oT^i)x_{n}||  =  \lim_{n\rightarrow \infty}||Jx_{n}-T^ix_{n}||=0,  ~\forall~ T^i\in \Gamma, $$  which implies
$ ||Jx^{*}-Jy^{*}||= ||Jx^{*}-J(J_{*}oT^i)x^*||= ||Jx^*-T^ix^*||=0 .$ So, $x^*\in F_{J}(\Gamma)$ for each $i$.


\noindent Next, let $u_n=T_{r_n}y_n$ for all $n\in \mathbb{N}$. Also, from (\ref{limun}), $u_n\rightarrow x^*~as~n\rightarrow \infty$. 
From Lemma \ref{ww4} and inequality (\ref{key2}), we have
{\bl\begin{eqnarray*}
 \phi(u_{n},y_n)  &=& \phi(T_{r_n}y_n,y_n)\\
                  &\leq&\phi(u,y_n) - \phi(u,T_{r_n}y_n)\nonumber\\
                  &\leq&\phi(u,x_n) - \phi(u,u_n)\nonumber
                   \end{eqnarray*}}
Since {\bl$\lim_{n\rightarrow \infty}(\phi(u,x_n) - \phi(u,u_n))=0,$} we have that {\bl$\lim_{n\rightarrow \infty}\phi(u_n,y_n)=0.$} From Lemma \ref{man},
we have that ${\bl\lim}_{n\rightarrow \infty}||y_n-u_n||=0.$ Again, since $r_n\in [a,\infty)$ and $J$ is uniformly continuous on bounded subsets of $\mathcal{X}$, we have that 
\begin{eqnarray}\label{223}
 {\bl\lim}_{n\rightarrow \infty}\frac{||Jy_n-Ju_n||}{r_n}=0.
\end{eqnarray}
{\bl Let $k\in \mathbb{N}$ and let $f_i:J\Omega\times J\Omega\to \mathbb{R}$ be a bifunction satisfying $(A1)$--$(A4)$ for each\\ $i\in\{1,2,\ldots,k\}$. Define
$F:J\Omega\times J\Omega\to \mathbb{R}$ by
\begin{equation}
    F(u,v)=\sum_{i=1}^{k} f_i(u,v), \qquad \forall\, u,v\in J\Omega.
\end{equation}
Equivalently, for all $x,y\in\Omega$,
\begin{equation}
    F(Jx,Jy)=\sum_{i=1}^{k} f_i(Jx,Jy).
\end{equation}
Since $F$ is a finite sum of bifunctions satisfying $(A1)$--$(A4)$, it also satisfies $(A1)$--$(A4)$.}\\
From $u_n=T_{r_n}y_n$, we have that
$$F(Ju_n,Jy)+\frac{1}{r_n}\langle u_n-y_{n}, Jy-Ju_n\rangle\geq0,~~\forall~~y\in \Omega.$$
By (A2), we have 
\begin{eqnarray}\label{224}
 \frac{1}{r_n}\langle u_n-y_{n}, Jy-Ju_n\rangle\geq-F(Ju_n,Jy)\geq F(Jy,Ju_n),~~\forall~~y\in \Omega.
\end{eqnarray}
Since $F(x,\cdot)$ is convex and lower semicontinuous and $u_n\rightarrow x^*$, it follows from equation (\ref{223}) and inequality (\ref{224}) that
$$F(Jy,Jx^*)\leq0,~~\forall~~y\in \Omega.$$
For $t\in(0,1]$ and $y\in \Omega$, let $y^*_t=tJy+(1-t)Jx^*$. Since, $J\Omega$ is convex, we have that $y^*_t\in J\Omega$ and hence $F(y^*_t,Jx^*)\leq0$. 
From (A1),
$$0=F(y^*_t,y^*_t)\leq tF(y^*_t,Jy)+(1-t)F(y^*_t,Jx^*)\leq tF(y^*_t,Jy),~~\forall~~y\in \Omega.$$ This implies that
$$F(y^*_t,Jy)\geq 0,~~\forall~~y\in \Omega.$$
Letting  $t\downarrow0$, from (A3), 
$$F(Jx^*,Jy)\geq 0,~~\forall~~y\in \Omega.$$
Therefore, we have that $Jx^*\in JGMEP(f).$ This implies that $x^*\in GMEP(f).$

\noindent{\bf (VI)} Finally, we show that $x^* = R_{B}x.$\\
From Lemma \ref{lem2}$(ii)$, we obtain that
\begin{equation}\label{eq4.3}
\phi(x,R_Bx) \leq \phi(x,x^*) - \phi(R_Bx, x^*) \leq  \phi(x,x^*).
\end{equation}

\noindent Again, using Lemma \ref{lem2}$(ii)$, definition of $x_{n+1}$, and $ x^*\in B \subset \Omega_{n},$ we compute as follows:
\begin{eqnarray}
 \phi(x,x_{n+1}) &\leq &\phi(x,x_{n+1}) + \phi(x_{n+1},R_Bx)\nonumber\\
                  & = & \phi( x, R_{\Omega_{n+1}}x) +  \phi(R_{\Omega_{n+1}}x, R_Bx) \leq \phi(x, R_Bx).\nonumber
 \end{eqnarray}
Since $x_{n}\rightarrow x^*$, taking limits on both sides of the last inequality, we obtain
\begin{equation}\label{eq4.4}
\phi(x,x^*) \leq \phi(x, R_Bx).
\end{equation}
Using inequalities (\ref{eq4.3}) and (\ref{eq4.4}), we obtain that $\phi(x,x^*) = \phi(x, R_Bx)$. By the 
uniqueness of $R_B$ ({\bl Lemma} \ref{ww1}), we obtain that $x^*=R_Bx$. This completes proof of the theorem. 
\end{proof}
 

\begin{theorem}\label{main2}
Let $\mathcal{X}$ be a uniformly smooth and uniformly convex real Banach space  with dual space $\mathcal{X}^*$ and let $\Omega$ be a nonempty closed and convex subset of $\mathcal{X}$ 
such that $J\Omega$ is closed and convex. Let $\varphi:J\Omega \rightarrow \mathbb{R}$ be a lower semi-continuous and convex function. For each $i \in \{1, 2, 3, ..., N\}$, {\bl let $f_{i}$ be a bifunction} from $J\Omega\times J\Omega$ to $\mathbb{R}$ satisfying $(A1)-(A4)$, 
$T^i_{n}:\Omega\rightarrow \mathcal{X}^*, n=1, 2, 3, ...$ be an infinite family of generalized $J-$nonexpansive maps and $\Gamma$ be a family of closed and generalized $J-$nonexpansive maps from $\Omega$ to $\mathcal{X}^*$ such that $\cap_{n=1}^{\infty}F_{J}(T^i_{n})=F_{J}(\Gamma) \neq \emptyset$
and $B := F_{J}(\Gamma)\cap GMEP(f,A, \varphi) \neq \emptyset.$ Assume that $\{T_{n}\}$ satisfies the NST-condition with $\Gamma$.
Let $\{x_{n}\}$ be generated by:  

\begin{equation}\label{algcoro2}
\begin{cases}  & x_{1} = x\in \Omega; \Omega_{1}=\Omega, \cr
                     & y_n =\alpha^0_{n}x_{n} + \sum_{i = 1} ^{N} \alpha^i_{n}J^{-1}oT^i_{n}x_{n},\cr
                    & u_{n}\in \Omega, ~~such~~that~~ \sum_{i = 1} ^{N} f_i(Ju_n,Jy)+\varphi(Jy) - \varphi (Ju_n) 
                    \cr & + \langle y - u_n, A u_n \rangle + \frac{1}{r_n}\langle u_n-y_{n}, Jy-Ju_n\rangle\geq0,~~\forall~~y\in \Omega,\cr
                    & \Omega_{n+1} =\{z\in \Omega_{n} : \phi(u_{n}, z) \leq \phi(x_{n}, z)\},\cr
                     & x_{n+1} = R_{\Omega_{n+1}}x,
\end{cases} 
\end{equation}

for all $n\in \mathbb{N}, \; \{\alpha^i_{n}\}\in [0,1]$ such that $\sum_{i = 0} ^{N} \alpha^i_{n} = 1$,  $\{r_n\}\subset [a,\infty)$
 for some $a>0$. Then, $\{x_{n}\}$ converges strongly 
to $R_B x$, where $R_B$ is the sunny generalized $J-$nonexpansive retraction of $\mathcal{X}$ onto $B$.   
\end{theorem}

\begin{proof}
 \noindent It is easy to see that $\{x_{n}\}$ is well defined.\\

\noindent Clearly, $B \subset \Omega_1$. Suppose  $B \subset \Omega_{n}$ for some $n\in \mathbb{N}$. Let $u\in B$, and $ u_n=T_{r_n}y_n$ for all $n\in \mathbb{N}$. Using the fact that {\bl$\{T^i_{n}\}$}
is an infinite family of generalized $J-$nonexpansive maps, the definition of $y_n$, Lemmas \ref{ww4}, and  \ref{zhang}, we compute as follows:
\begin{eqnarray}
 \phi(u_{n},u) & = & \phi(T_{r_n}y_n,u)\leq\phi(y_n,u)  = \phi{\bl\Big(}\alpha^0_{n}x_{n} + \sum_{i = 1} ^{N}                         \alpha^i_{n}J^{-1}oT^i_{n}x_{n},u{\bl\Big)}\nonumber\\
                & = & {\bl\Big|\Big|}\alpha^0_{n}x_{n} + \sum_{i = 1} ^{N}                         \alpha^i_{n}J^{-1}oT^i_{n}x_{n}{\bl \Big|\Big|}^2 -2{\bl\Big\langle} \alpha^0_{n}x_{n} + \sum_{i = 1} ^{N}                         \alpha^i_{n}J^{-1}oT^i_{n}x_{n},Ju{\bl\Big\rangle} + ||u||^2\nonumber\\
                & \leq & \alpha^0_{n}||x_{n}||^2 -2\alpha^0_{n}\langle x_{n}, Ju\rangle + ||u||^2 + \sum_{i = 1} ^{N} \alpha^i_{n}||J^{-1}oT^i_{n}x_{n}||^2  -2{\bl\Big\langle} \sum_{i = 1} ^{N} \alpha^i_{n}J^{-1}oT^i_{n}x_{n},Ju{\bl\Big\rangle}\nonumber\\
                && -  \alpha^0_{n}\alpha^i_{n}g(||x_{n}- J^{-1}oT^i_{n}x_{n}||)\nonumber\\
                & = & \alpha^0_{n}(||x_{n}||^2 -2\langle x_{n}, Ju\rangle + ||u||^2) + \sum_{i = 1} ^{N}\alpha^i_{n}(||J^{-1}oT^i_{n}x_{n}||^2 -2\langle J^{-1}oT^i_{n}x_{n}, Ju\rangle  + ||u||^2) \nonumber\\ 
                && -  \alpha^0_{n}\alpha^i_{n}g(||x_{n}- J^{-1}oT^i_{n}x_{n}||)\nonumber\\
                & = & \alpha^0_{n}\phi(x_{n},u) + \sum_{i = 1} ^{N}\alpha^i_{n}\phi(J^{-1}oT^i_{n}x_{n},u) -  \alpha^0_{n}\alpha^i_{n}g(||x_{n}- J^{-1}oT^i_{n}x_{n}||)\nonumber\\
                & \leq & \alpha^0_{n}\phi(x_{n},u) + \sum_{i = 1} ^{N}\alpha^i_{n}\phi(x_{n},u) -  \alpha^0_{n}\alpha^i_{n}g(||x_{n}- J^{-1}oT^i_{n}x_{n}||).\nonumber
\end{eqnarray}
Hence, we obtain the following inequality for each $i \in \{1,2,3,\cdots, N\}$
\begin{equation}\label{now}
\phi(u_{n},u) \leq \phi(x_{n},u) -  \alpha^0_{n}\alpha^i_{n}g(||x_{n}- J^{-1}oT^i_{n}x_{n}||).
\end{equation}
Additionally, we have that 
\begin{equation}
\phi(y_{n},u) \leq \phi(x_{n},u) ~\forall~ n\in \mathbb{N}.
\end{equation}
From inequality {\bl(\ref{now})}, we conclude that $u\in \Omega_{n+1}$. Thus, $B \subset{\Omega_n} $ for all $n \ge 1$.\\

\noindent The rest of the proof is similar to the proof of Theorem \ref{main}.

\end{proof}

\begin{example}
Let $\mathcal{X} = l_p$, $1< p <\infty$, $\frac{1}{p} + \frac{1}{q} = 1$, and $\Omega = \overline{B_{l_p}}(0,1)$ = $\{x \in l_p : ||x||_{l_p}\leq 1\}$. Then $J{\bl\Omega} = \overline{B_{l_q}}(0,1)$. Let $f_i : J\Omega\times J\Omega \longrightarrow \mathbb{R}$ defined by $f_i(x^*, y^*) = \langle J^{-1}x^*, y^* - x^*\rangle$ {\bl$\forall$ $x^*,y^* \in J\Omega$} and for each $i \in \{1,2,3,\cdots,k\}$, $A : \Omega \longrightarrow l_q$ defined by ${\bl A}x = J(x_1, x_2, x_3, \cdots)$ $\forall$ $x = (x_1, x_2, x_3, \cdots) \in \Omega$, $\varphi : J\Omega \rightarrow \mathbb{R}$ defined by $\varphi(x^*) = ||x^*||, \forall~ x^* \in J\Omega$, $T : \Omega \longrightarrow l_q$ defined by $Tx = J(0, x_1, x_2, x_3, \cdots)$ $\forall$ $x = (x_1, x_2, x_3, \cdots) \in \Omega$, {\bl $\Gamma = \{T\}$}, and $T_n : \Omega \longrightarrow l_q$ defined by $T_{n}x = \alpha_nJx + (1 - \alpha_n)Tx, ~\forall n \geq 1,~\forall~ x\in \Omega, \alpha_n \in (0,1) \text{ such that } 1 - \alpha_n \ge \frac{1}{2}$. Then $\Omega$, $J\Omega$, $f_i$, $A$, $\varphi$, $T$, and $T_n$ satisfy the conditions of Theorems \ref{main} and \ref{main2}. Moreover, $0 \in F_{J}(\Gamma) \cap GMEP(f,A,\varphi)$.
\end{example}
\noindent Let $A = 0$ in Theorems \ref{main} and \ref{main2}. We obtain the following results.
 \begin{corollary}
  Let $\mathcal{X}, \mathcal{X}^*, \Omega, \varphi, \{f_i\}, \{T^i_{n}\}, \{r_n\}, \{\alpha^i_{n}\}$ be as in Theorem \ref{main2}. If $\cap_{n=1}^{\infty}F_{J}(T^i_{n})=F_{J}(\Gamma) \neq \emptyset$
and $B := F_{J}(\Gamma)\cap MEP(f,A, \varphi) \neq \emptyset$, where $MEP$ is the set of solutions to mixed equilibrium problem \eqref{mep}. Let $\{x_{n}\}$ be generated by:  

\begin{equation}\label{algcoro1}
\begin{cases}  & x_{1} = x\in \Omega; \Omega_{1}=\Omega, \cr
                     & y_n =\alpha^0_{n}x_{n} + \sum_{i = 1} ^{N} \alpha^i_{n}J^{-1}oT^i_{n}x_{n},\cr
                    & u_{n}\in \Omega, ~~such~~that~~ \sum_{i = 1} ^{N} f_i(Ju_n,Jy)+\varphi(Jy) - \varphi (Ju_n) 
                    + \frac{1}{r_n}\langle u_n-y_{n}, Jy-Ju_n\rangle\geq0,~~\forall~~y\in \Omega,\cr
                    & \Omega_{n+1} =\{z\in \Omega_{n} : \phi(z, u_{n}) \leq \phi(z, x_{n})\},\cr
                     & x_{n+1} = R_{\Omega_{n+1}}x,
\end{cases} 
\end{equation}

for all $n\in \mathbb{N}$. Then, $\{x_{n}\}$ converges strongly 
to $R_B x$.   
 \end{corollary}
\section{Applications}
\subsection{Applications in classical Banach spaces}
\vspace{0.2cm}
\noindent Theorems \ref{main} and \ref{main2} are applicable in classical Banach spaces, such as $L_{p},~l_{p}, or~ W^{m}_{p}(\Omega)$, $p \in (1,\infty)$, where $W^{m}_{p}(\Omega)$ denotes the usual Sobolev space. The analytical representations of duality maps are known in $L_p,$ $l_p,$ and
$W^p_m(\Omega),$ $p \in (1,\infty)$, $p^{-1}+q^{-1}=1$, see e.g., \cite{yak}. 

\subsection{Applications in Hilbert spaces}
\begin{corollary}\label{app1}
{\bl Let $\mathcal{X}=H$ be a real} Hilbert space and let $\Omega$ be a nonempty closed and convex subset of $H$. Let $\varphi:\Omega \rightarrow \mathbb{R}$ be a lower semi-continuous and convex function. For each $i \in \{1, 2, 3, ..., N\}$, {\bl let $f_{i}$ be a bifunction} from $\Omega\times \Omega$ to $\mathbb{R}$ satisfying $(A1)-(A4)$, 
$T^i_{n}:\Omega\rightarrow H, n=1, 2, 3, ...$ be an infinite family of nonexpansive maps and $\Gamma$ be a family of closed and nonexpansive maps from $\Omega$ to $H$ such that $\cap_{n=1}^{\infty}F(T_{n})=F(\Gamma) \neq \emptyset$
and $B := F(\Gamma)\cap GMEP(f,A, \varphi) \neq \emptyset.$ Assume that $\{T_{n}\}$ satisfies the NST-condition with $\Gamma$.
Let $\{x_{n}\}$ be generated by:  

\begin{equation}\label{algcoro3}
\begin{cases}  & x_{1} = x\in \Omega; \Omega_{1}=\Omega, \cr
                     & y_n =\alpha^0_{n}x_{n} + \sum_{i = 1} ^{N} \alpha^i_{n}T^i_{n}x_{n},\cr
                    & u_{n}\in \Omega, ~~such~~that~~ \sum_{i = 1} ^{N} f_i(u_n,y)+\varphi(y) - \varphi (u_n) 
                    \cr & + \langle y - u_n, A u_n \rangle + \frac{1}{r_n}\langle u_n-y_{n}, y-u_n\rangle\geq0,~~\forall~~y\in \Omega,\cr
                    & \Omega_{n+1} =\{z\in \Omega_{n} : ||u_{n} - z|| \leq ||x_{n} - z||\},\cr
                     & x_{n+1} = P_{\Omega_{n+1}}x,
\end{cases} 
\end{equation}

for all $n\in \mathbb{N}, \; \{\alpha^i_{n}\}\in [0,1]$ such that $\sum_{i = 0} ^{N} \alpha^i_{n} = 1$,  $\{r_n\}\subset [a,\infty)$
 for some $a>0$. Then, $\{x_{n}\}$ converges strongly 
to $P_B x$, where $P_B$ is the metric projection of $H$ onto $B$.   
\end{corollary}

\begin{proof}
In a Hilbert space, $J$ is the identity operator and $\phi(x,y)=||x-y||^{2} ~ \text{for all} ~ x,y\in H$. The result follows from Theorem \ref{main} or Theorem \ref{main2}.
\end{proof}

\subsection{Applications to optimization problem}
Consider the following optimization problem:
\begin{equation}\label{appoptm}
    \min_{x \in \Omega}(\psi(x) + \varphi(x))
\end{equation}
where $\Omega$ is a nonempty closed convex subset of a Hilbert space $H$, and $\psi, \varphi :\Omega\rightarrow \mathbb{R}$ are two
convex and lower semi-continuous functionals.  Let $\Lambda \subset \Omega$ be the set of solutions to \eqref{appoptm}. Clearly, $\Lambda$ is a closed convex subset of $\Omega$. Let $f : \Omega \times \Omega \rightarrow \mathbb{R}$ be a bifunction
defined by $f(x, y) = \psi(y) - \psi(x)$. Consider the following mixed equilibrium problem: find $x^* \in \Omega$ such that
\begin{equation}\label{appmep}
    f(x^*,y)+\varphi(y) - \varphi (x^*) \geq 0,~~\forall~~y\in \Omega.
\end{equation}
Then, $f$ satisfies conditions (A1)–(A4) and $MEP = \Lambda$, where $MEP$ is the set of solutions to mixed equilibrium problem \eqref{appmep}. Let $\{x_n\}$ be the iterative sequence
generated by: 

\begin{equation}\label{algcoro3}
\begin{cases}  & x_{1} = x\in \Omega, \,\Omega_{1}=\Omega; \cr
                     & x_{n+1} = P_{\Omega_{n+1}}x,\, \Omega_{n+1} =\{z\in \Omega_{n} : ||u_{n} - z|| \leq ||x_{n} - z||\};\cr
                     & u_{n}\in \Omega, ~~such~~that~~ f(u_n,y)+\varphi(y) - \varphi (u_n) 
                     + \frac{1}{r_n}\langle u_n-y_{n}, y-u_n\rangle\geq0,~~\forall~~y\in \Omega,
\end{cases} 
\end{equation}
for all $n\in \mathbb{N}$,  $\{r_n\}\subset [a,\infty)$
 for some $a>0$, where $P_\Omega$ is the metric projection of $H$ onto $\Omega$. Then, $\{x_{n}\}$ converges strongly 
to $P_\Omega x$.

\end{document}